\documentclass[11pt]{article}

\usepackage[utf8]{inputenc}
\usepackage{amsfonts,amsthm}
\usepackage{amsmath}   
\usepackage{latexsym}
\usepackage{amssymb}
\usepackage{mathrsfs}
\usepackage{pifont}
\usepackage{tikz}
\usepackage[all]{xy}
\usepackage[hidelinks]{hyperref}

\newtheoremstyle{rem}{3pt}{3pt}{}{}
{\bfseries}{.}{.5em}{}

\newtheorem{theo}{Theorem}[section]
\newtheorem*{theo*}{Theorem}
\newtheorem{defi}[theo]{Definition}

\newtheorem{prop}[theo]{Proposition}

\newtheorem{coro}[theo]{Corollary}

\newtheorem{rema}[theo]{Remark}

\theoremstyle{rem}

\newcommand{\diag}[3]{ \foreach \t in {1,...,#3} {\draw[thick] (#1+\t,#2-1) rectangle (#1+\t-1,#2);} }
\newcommand{\diagg}[4]{ \foreach \t in {1,...,#3} {\draw[thick] (#1+\t,#2-1) rectangle (#1+\t-1,#2);} \foreach \t in {1,...,#4} {\draw[thick] (#1+\t,#2-1) rectangle (#1+\t-1,#2-2);} }

\usepackage{geometry}
\newcommand{\bC}{\mathbb{C}}
\newcommand{\bZ}{\mathbb{Z}}
\newcommand{\cP}{\mathcal{P}}
\newcommand{\cPC}{\mathcal{C}}
\newcommand{\cPM}{\mathcal{M}}
\newcommand{\rep}[1]{\mathbf{\underline{#1}}}

\title{Pascal, Catalan, Motzkin triangles and tensor product multiplicities}

\author{L. Poulain d'Andecy\footnote{Laboratoire de math\'ematiques de Reims, UMR CNRS 9008, Universit\'e de Reims Champagne-Ardenne, 51100 Reims, France. \emph{email adress}: loic.poulain-dandecy@univ-reims.fr}}

\date{}

\begin{document}

\maketitle

\begin{abstract}
The main purpose of this note is to provide an elementary discussion of some simple triangles of integer numbers in particular through their connections with representation theory of $sl_2$. The triangles under consideration are the Catalan triangle and the Motzkin triangle together with their generalisations that we introduce here. We advocate the point of view that these triangles are given by the well-known and classical Pascal rule starting from a well-chosen initial condition. We give an elementary derivation of the fact that the numbers in these triangles are multiplicities appearing in tensor products of $sl_2$-representations and that they are simply expressed as a difference of generalised binomial coefficients. We also take the opportunity to discuss the ``sum of squares'' phenomenon that happens in these triangles through the lense of representation theory.
\end{abstract}

\section{Introduction}

The first triangle under consideration in this paper is the following:
\begin{equation}\label{PC2}
\cPC_2=\quad\begin{array}{cccccccccccccccccccc}
&&& -1\  && 1 &&&&&&&&&&&& \leadsto& 1\\
&& -1\  && 0 && 1 &&&&&&&&&&& \leadsto& 1\\
& -1 && -1\  && 1 && 1 &&&&&&&&&& \leadsto& 2\\[-0.4em]
\reflectbox{$\ddots$} && -2\  && 0 && 2 && 1 &&&&&&&&& \leadsto& 5\\[-0.4em]
& \vdots && -2\  && 2 && 3 && 1 &&&&&&&& \leadsto& 14\\[-0.4em]
&&\vdots&& 0 && 5 && 4 && 1&&&&&&& \leadsto& 42\\[-0.4em]
&&&\vdots&& 5 && 9 && 5 && 1&&&&&& & \vdots\\
&&&& 0 && 14 && 14 && 6 && 1\\
&&&&& 14 && 28 && 20 && 7 && 1\\
&&&& 0 && 42 && 48 && 27 && 8 && 1\\
&&&&& 42 && 90 && 75 && 35 && 9 && 1\\
&&&& 0 && 132 && \ldots && && && && \ddots \\
&&&&& 132\!\! && && \ldots
\end{array}
\end{equation}
It is defined by the usual Pascal rule (each entry is the sum of the two entries just above) starting from the initial line with two vertices $-1$ and $1$. This triangle appeared at least once in the literature in \cite{Kir95} (in a different context than here) and is found in A112467 in OEIS \cite{OEIS}.

The positive half of the triangle (sometimes arranged differently) is usually found under the name of \emph{Catalan triangle}, see A008313, A008315, A009766, A053121 in \cite{OEIS}. Indeed one can not help to notice that the Catalan numbers appear in the first column of the positive part of the triangle. Less obvious is that when summing the squares of the positive coefficients of each line, one recovers also the Catalan numbers. This is what is indicated by the $\leadsto$ at the end of each line. This remarkable coincidence will happen for all the triangles in the paper and we will give both a combinatorial and a representation-theoretic explanation of this fact.

Note that the above triangle has an obvious symmetry such that the left half is the opposite of the right half. Many of our triangles will be like that and therefore we will usually not show much of the negative side, to save space.

One could argue that the negative side of the triangle is not necessary, and that one could define the positive half of the triangle starting with a single $1$ and enforcing the usual Pascal rule together with the additional restriction of having a first column of $0$. We now explain some advantages of the version above with the negative numbers.

First of all, the version in (\ref{PC2}) is quite suitable for an easy generalisation, where we simply keep the same initial condition and replace the usual Pascal rule by its higher analogues (see Section \ref{sec_Cat} for more precision). For example, here is the analogue obtained by summing three entries at each step:
\begin{equation}\label{PC3}
\cPC_3=\quad\begin{array}{ccccccccccccccccc}
&&& -1 & 1 &&&&&&&& \leadsto& 1\\
& & -1 & 0 & 0 & 1 &&&&&&& \leadsto& 1\\
& -1 & -1 & -1 &1 & 1 & 1&&&&&& \leadsto& 3\\[-0.4em]
\reflectbox{$\ddots$}&&&-1 & 1 & 3 & 2 & 1&&&&& \leadsto& 15\\[-0.4em]
&&\vdots&-3 & 3 & 6 & 6 & 3 & 1&&&& \leadsto& 91\\[-0.4em]
&\vdots&&-6 & 6 & 15 & 15 & 10 & 4 & 1&&&& \vdots\\
&&&-15& 15 & 36 & 40 & 29 & 15 & 5 & 1\\[-0.4em]
&&&-36& 36 & 91 & &\ldots &&&&\ddots\\
&&&-91& 91 & &\ldots 
\end{array}
\end{equation}
In that case, it would still be possible to forget the negative half and construct only the positive half but it would be somewhat less convenient. And this will be more and more so when going to higher and higher analogues. The (positive half) of the triangle (\ref{PC3}) is A089942 in OEIS and the first positive column are the Riordan numbers, see A005043 \cite{OEIS}. The full triangle with its negative half does not seem to appear in OEIS. Note again the ``sum of squares'' phenomenon involving the first positive coefficients of the even lines.

The second argument and our main reason to considering the triangle (\ref{PC2}) (and its generalisations like (\ref{PC3})) is that it provides directly a very simple interpretation of its numbers. Indeed, consider the following formula:
\[(q-q^{-1})(q+q^{-1})^n=\sum_{k\geq0}c_{k,n}(q^k-q^{-k})\ .\]
It should be almost obvious that the integers $c_{k,n}$ appearing here are exactly the integers appearing in the triangle (\ref{PC2}) above (we will need to fix our conventions about the indices, and we will do that more carefully later). Indeed the initial condition corresponds to $(q-q^{-1})$ while the Pascal rule corresponds to the successive multiplications by $(q+q^{-1})$. Therefore, a line of the triangle is simply recording the coefficients in front of the powers of $q$ in the polynomials $(q-q^{-1})(q+q^{-1})^n$ hence the negative coefficients and the antisymmetry.

Then if we divide both sides by $(q-q^{-1})$ in the preceding formula, we get:
\begin{equation}\label{expansion_intro}(q+q^{-1})^n=\sum_{k\geq0}c_{k,n}[k]_q\,,\end{equation}
where $[k]_q=\frac{q^k-q^{-k}}{q-q^{-1}}$ are the usual $q$-numbers. This is probably the simplest interpretation (and could be a definition) of the integers $c_{k,n}$. Note that the well-known Catalan numbers are $c_{1,2n}$ and are the coefficients in front of $[1]_q=1$ when $(q+q^{-1})^{2n}$ is expanded in terms of the $q$-numbers $[k]_q$.

Now, considering the characters of $sl_2$-representation, Formula (\ref{expansion_intro}) implies immediately that the integers $c_{k,n}$ are the multiplicities appearing in the tensor powers of the $2$-dimensional representation of $sl_2$ (see Section \ref{sec_Cat} for the notations and more details):
\begin{equation}\label{tensorproduct_intro}
\rep{2}^{\otimes n}\cong\bigoplus_{k>0} c_{k,n}\rep{k}\,.\end{equation}
Of course, only the positive half of the triangle represents multiplicities, and one may see the negative half as the trick to get the correct positive integers from a simple Pascal rule. What we just sketched we see as the shortest way to calculate the multiplicities appearing in (\ref{tensorproduct_intro}) and this generalises verbatim to tensor powers of any finite-dimensional irreducible representation of $sl_2$. These multiplicities regularly appear in the physics literature because representations of $sl_2$ are also representations of $SU(2)$, see for example \cite{CHI15,CKZ17} for recent occurrences.

Finally, another interesting aspect of the triangle (\ref{PC2}) with its negative sides is that it should make almost obvious that the coefficients $c_{k,n}$ are differences of two binomial coefficients. Indeed, one only has to look closely for a little while to see that the triangle is a difference of two Pascal triangles, one slightly shifted from the other.

We will discuss all these features of the Catalan triangle and its higher analogues with more details and more precision in the paper. Moreover, we will see that the same story applies for another famous triangle of numbers, namely the Motzkin triangle. Here is how we find the Motzkin triangle in our context:
\begin{equation}\label{PM3}
\cPM_3=\quad\begin{array}{ccccccccccccccccc}
&&& -1 &0& 1 &&&&&&&& \leadsto& 1\\
&& -1 & -1 & 0 & 1 & 1 &&&&&&& \leadsto& 2\\
&-1&-2&-2& 0 & 2 & 2 & 1 &&&&&& \leadsto& 9 \\[-0.5em]
\reflectbox{$\ddots$}&&\vdots& & 0 & 4 & 5 & 3 & 1&&&&& \leadsto& 51\\[-0.5em]
&&&&\vdots& 9 & 12 & 9 & 4 & 1&&&& & \vdots\\[-0.5em]
&&&&\vdots& 21 & 30 & 25 & 14 & 5 & 1\\[-0.5em]
&&&&& 51 & 76 & \ldots &&&& \ddots
\end{array}
\end{equation}
This triangle appears in A349812 in OEIS \cite{OEIS}. It is simply given by the Pascal rule of order $3$ applied to an initial condition $-1,0,1$. The positive half of it is the Motzkin triangle, see A026300 or A064189 in OEIS \cite{OEIS}, the first positive column of the triangle giving the Motzkin numbers, A001006 in \cite{OEIS}. Note again the ``sum of squares'' phenomenon. The point of view taken in this paper (Pascal rule with well-chosen initial conditions and $sl_2$-representation theory) applies equally well for the Motzkin triangle and provides immediately natural higher analogues, which seem to be new. We refer to the last section of the paper for more details and more examples.

\paragraph{Organisation of the paper.} In Section \ref{sec_Pas}, as a warm-up and for fixing conventions, we deal with the Pascal triangle and its higher analogues. This allows us to introduce the generalised binomial coefficients that will be used all along, and to show in a toy model example how a triangle is to be interpreted in terms of tensor product multiplicities. We also show in this example the first occurences of the ``sum of squares'' phenomenon.

In Section \ref{sec_Cat}, we study the Catalan triangle and its generalisations (with their negative sides of course). First we give the simple and natural interpretation as an expansion in terms of $q$-numbers, as in (\ref{expansion_intro}) above. Then we quickly deduce the fact that these triangles give multiplicities of tensor products of $sl_2$ representations. Finally, we emphasise the nice ``sum of squares'' property already mentioned and prove it combinatorially and representation-theoretically.

In Section \ref{sec_Mot}, we show how to include the Motzkin triangle in the same sort of discussion. All interpretations and properties of the Catalan triangles already discussed are generalised to the Motzkin triangle and its higher analogues.

\paragraph{Notations.} Throughout the paper, $q$ is an indeterminate and we use the following notation for the $q$-numbers:
\[[k]_q=\frac{q^k-q^{-k}}{q-q^{-1}}=q^{1-k}+q^{3-k}+\dots+q^{k-3}+q^{k-1}\ .\]

\section{Pascal triangles}\label{sec_Pas}

\subsection{Definition}

\paragraph{The usual Pascal triangle.} The usual Pascal triangle is
\begin{equation}\label{P2}
\cP_2= \begin{array}{ccccccccccccc}
&&&\ldots&\scriptstyle{(-2)} & \scriptstyle{(-1)} & \scriptstyle{(0)}& \scriptstyle{(1)} & \scriptstyle{(2)} &  \ldots\\[0.6em]
&&&&&& 1 \\
&&&&& 1 && 1 \\
&&&&1 && 2 && 1 \\
&&&1&& 3 && 3 && 1 \\
&&1 && 4 && 6 && 4 && 1 \\[-0.4em]
&\reflectbox{$\ddots$} &&\vdots& &10&& 10 & & \vdots && \ddots\\
&&&&&& 20
\end{array}
\end{equation}
It is given by a single initial vertex $1$ and the rule that the value of an entry is obtained by adding the two entries just above. More precisely, the columns of the triangle are labelled as shown above it and the lines are labelled naturally by non-negative  integers $n$ with the first line corresponding to $n=0$. Then the initial condition and the recurrence are given by:
\[b_{k,0}^{(2)}=\delta_{k,0}\ \ \ \ \ \text{and}\ \ \ \ \ \ b_{k,n+1}^{(2)}=b_{k-1,n}^{(2)}+b_{k+1,n}^{(2)}\ .\] 
By construction, the integers in the Pascal triangle are the coefficients appearing in the following expansion:
\begin{equation}\label{expansionP2}
[2]_q^n=(q^{-1}+q)^n=\sum_{k}b_{k,n}^{(2)}q^k\ .\end{equation}
In terms of binomial coefficients, the formula is $b_{k,n}^{(2)}=\binom{n}{\frac{k+n}{2}}$. Note that with our convention, we have $b_{k,n}^{(2)}=b_{-k,n}^{(2)}=\binom{n}{\frac{n-k}{2}}$, and that the entries are non-zero only for $k=-n,-n+2,\dots,n-2,n$ on the line $n$. Of course we do not display all the trivial $0$'s.

\paragraph{Higher Pascal triangles.} We fix an integer $d\geq 2$, and we generalise the expansion (\ref{expansionP2}) to
\begin{equation}\label{expansionPd}
[d]_q^n=(q^{1-d}+q^{3-d}+\dots+q^{d-3}+q^{d-1})^n=\sum_{k}b_{k,n}^{(d)}q^k\ ,
\end{equation}
thus defining the integers $b_{k,n}^{(d)}$. The initial condition is $[d]_q^0=1$, which is the same as before: $b_{k,0}^{(d)}=\delta_{k,0}$, and corresponds to a first line with a single non-zero entry $1$ in column $0$. The recursion is given by
\begin{equation}\label{higherPascalrule}b_{k,n+1}^{(d)}=b_{k-(d-1),n}^{(d)}+b_{k-(d-3),n}^{(d)}+\dots + b_{k+d-3,n}^{(d)}+b_{k+d-1,n}^{(d)}\ .\end{equation}
These numbers are calculated recursively by drawing the higher analogues of the Pascal triangle, where we start with the same initial line and use the rule of adding $d$ entries corresponding to (\ref{higherPascalrule}). Note that for $d$ odd, only even powers of $q$ will appear so we only need to show the columns labelled by even integers. For example, for $d=3$, here is the triangle:
\begin{equation}\label{P3}
\cP_3= \begin{array}{ccccccccccccc}
&&\ldots&\scriptstyle{(-4)}\!\! & \scriptstyle{(-2)} & \scriptstyle{(0)}& \scriptstyle{(2)} & \scriptstyle{(4)} &  \ldots\\[0.6em]
&&&&& 1 \\
&&&&1 &1& 1 \\
&&&1\!\!\!&2&3&2 & 1 \\
&&1&3\!\!\!&6&7&6&3& 1 \\[-0.4em]
&\reflectbox{$\ddots$} &&\vdots\!\!\!&16&19& 16 & \vdots&& \ddots
\end{array}
\end{equation}
We show again the labels of the columns to emphasise that we display only the columns labelled by even integers (the others are filled with $0$'s).

For $d$ even, only even powers of $q$ appear when $n$ is even while only odd powers of $q$ appear when $n$ is odd. Again, when drawing a triangle, we only indicate the entries which can be non-zero and this explains that for $d$ even the entries of line $n+1$ are located between two entries of line $n$. For example, for $d=4$, the triangle looks like
\begin{equation}\label{P4}
\cP_4=\qquad \begin{array}{ccccccccccccccccccccc}
&&&&&&&\ldots&\scriptstyle{(-2)}\!\! & \scriptstyle{(-1)} & \scriptstyle{(0)}& \scriptstyle{(1)} & \scriptstyle{(2)} &  \ldots\\[0.6em]
&&&&&&&&&& 1 \\
&&&&&&&1\! && 1 && 1 && 1\\
&&&&1\  &&2&&3\!\! && 4 && 3&& 2&& 1 \\
&1&&3&&6& &10\!&& 12 &&12 && 10 && 6 && 3&&1\\[-0.4em]
\reflectbox{$\ddots$} &&&&&&&&&\ldots\!\!& 44 & \ldots&&&&&&&&&\ddots
\end{array}
\end{equation}
Here, we sum 4 entries at each step. With our conventions, for arbitrary $d$, the recursion (\ref{higherPascalrule}) is reflected by the simple rule that to calculate the value of an entry we sum the $d$ numbers of the line above centered around the given entry.

\subsection{Multiplicities in tensor products}

There is a very simple way to interpret the integers in the Pascal triangles $\cP_d$ as multiplicities appearing in tensor products. Namely consider the Abelian group $\bZ$ and take a non-zero complex number $z$ sufficiently generic (for example, a transcendental number). For any $k\in\bZ$, consider the following one-dimensional representation:
\[L_k\ :\ \ \ \ \bZ\ni m\mapsto (z^k)^m\ .\]
Of course, knowing that $1\mapsto z^k$ is enough to determine $L_k$. For an integer $d\geq 2$, we construct the following $d$-dimensional representation as a direct sum:
\[V_d:=L_{1-d}\oplus L_{3-d}\oplus \ldots\oplus L_{d-3}\oplus L_{d-1}\ .\]
By construction in $V_d$ we have $1\mapsto z^{1-d}+z^{3-d}+\dots+z^{d-3}+z^{d-1}=[d]_z$. Then decomposing the tensor product $V_d^{\otimes n}$ into a direct sum of irreducible repesentations amounts to expanding $[d]^n_z$ as a Laurent polynomial in $z$. Since $z$ is generic, this is equivalent to the way we introduced the integers forming the Pascal triangle $\cP_d$ and therefore we have
\[V_d^{\otimes n}\cong \bigoplus_{k} b_{k,n}^{(d)}L_k\ .\]
In words, the integer $b_{k,n}^{(d)}$ is the multiplicity of the simple module $L_k$ in the decomposition of the representation $V_d^{\otimes n}$.

\subsection{Sums of squares}

In the triangles $\cP_d$, for every $d\geq 2$, we have the following remarkable equality.
\begin{prop}\label{propsquaresPascal}
We have:
\begin{equation}\label{sumofsquares}
b_{0,2n}^{(d)}=\sum_{k}(b_{k,n}^{(d)})^2\ .
\end{equation}
In words, the central coefficient of line $2n$ is equal to the sum of the squared coefficients of line $n$.
\end{prop}
For example, for $d=2$, this amounts to the following well-known equality for binomial coefficients:
\[\binom{2n}{n}=\sum_{k} \binom{n}{k}^2\ .\]
One can check the formula in the examples above of $\cP_2$, $\cP_3$ and $\cP_4$. For example, in $\cP_3$, the sums of squares give the sequence $b_{0,2n}^{(3)}=1,3,19,141,...$ (A082758 in \cite{OEIS}), while in $\cP_4$, we get the sequence $b_{0,2n}^{(4)}=1,4,44,580,...$ (A005721 in \cite{OEIS}).

\paragraph{A combinatorial proof.} Here is a simple ``combinatorial'' proof of Formula (\ref{sumofsquares}) using the Pascal rule. First we are going to imagine that we have oriented edges in our triangles connecting an entry to the $d$ entries below it to which it contributes. Here is for example for $d=2$ the beginning of the graph
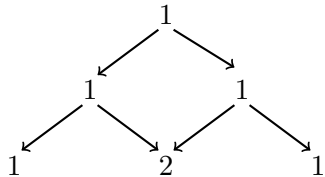
\begin{center}
\begin{tikzpicture}[scale=0.5]
\node at (0,0) {1};
\draw[thick,->] (-0.2,-0.4) -- (-1.8,-1.6);
\draw[thick,->] (0.2,-0.4) -- (1.8,-1.4);
\node at (-2,-2) {1};\node at (2,-2) {1};
\draw[thick,->] (-2.2,-2.4) -- (-3.8,-3.6);
\draw[thick,->] (-1.8,-2.4) -- (-0.2,-3.6);
\draw[thick,->] (1.8,-2.4) -- (0.2,-3.6);
\draw[thick,->] (2.2,-2.4) -- (3.8,-3.6);
\node at (-4,-4) {1};\node at (0,-4) {2};\node at (4,-4) {1};
\end{tikzpicture}
\end{center}
Thus we can speak of paths going from one entry to another by following these oriented edges. Let us denote by $|p_{x\to y}|$ the number of paths from a vertex $x$ to a vertex $y$ and let us call $\mathbf{1}$ the initial vertex of the graph. Then note that, by the very definition of the triangle, the value of an entry at a certain vertex is equal to the number of paths from $\mathbf{1}$ to this vertex, that is we have
\[b^{(d)}_{k,n}=|p_{\mathbf{1}\to (k,n)}|\ .\]
Now we can make the simple following calculation
\begin{equation}\label{soqproof}
|p_{\mathbf{1}\to (0,2n)}|=\sum_{\text{$x$ in line $n$}}|p_{\mathbf{1}\to x}||p_{x\to(0,2n)}|=\sum_{\text{$x$ in line $n$}}|p_{\mathbf{1}\to x}|^2\ .
\end{equation}
For the second equality we use the simple fact that, by symmetry, for any vertex $x$ in level $n$, we have that the number of paths from $x$ to the central vertex $(0,2n)$ is equal to the number of paths from $\mathbf{1}$ to $x$. Formula (\ref{soqproof}) is exactly Formula (\ref{sumofsquares}).

\paragraph{A representation-theoretic proof.} The ``sum of squares'' formula in Proposition \ref{propsquaresPascal} has also a simple interpretation in representation-theoretic terms. First note that, by definition of the multiplicity, we have
\[b_{0,2n}^{(d)}=\dim \textrm{Hom}_{\bZ}(V_d^{\otimes 2n},triv)\,,\]
where $triv=L_0$ is the trivial representation. On the other hand, we have
\[\sum_{k}(b_{k,n}^{(d)})^2=\dim \textrm{End}_{\bZ}(V_d^{\otimes n})\ .\]
Indeed, by Schur Lemma, the centraliser $\textrm{End}_{\bZ}(V_d^{\otimes n})$ of the representation $V_d^{\otimes n}$ is a direct sum of matrix algebras of sizes the multiplicities appearing in the decomposition.

Now we can use the following reasoning:
\begin{equation}\label{soqrepproof}\begin{array}{rcl}
\textrm{End}_{\bZ}(V_d^{\otimes n}) & = & \textrm{Hom}_{\bZ}(V_d^{\otimes n},V_d^{\otimes n})\\[0.4em]
 & \cong & \textrm{Hom}_{\bZ}(V_d^{\otimes n}\otimes (V^{\star}_d)^{\otimes n},triv)\\[0.4em]
 & \cong & \textrm{Hom}_{\bZ}(V_d^{\otimes 2n},triv)\ .
\end{array}\end{equation}
We have been using the classical fact that for any two representations $M,N$, we have $\textrm{Hom}_{\bZ}(M,N)\cong \textrm{Hom}_{\bZ}(M\otimes N^\star,triv)$ (see for example \cite[page 5]{FH91}) where $N^{\star}$ denotes the contragredient representation of $N$. In our case, the contragredient representation amounts to simply replace the action of $m$ by the action of $-m$ and we see easily that
\[L_k^\star=L_{-k}\ \ \ \ \text{and therefore}\ \ V_d^\star\cong V_d\ .\]
This explains the last equality in (\ref{soqrepproof}). Now taking the dimensions in the isomorphism (\ref{soqrepproof}) recovers the desired numerical equality (\ref{sumofsquares}).

\section{Catalan triangles} \label{sec_Cat}

\subsection{Definition and examples}
We start with the formal definition of what we call the Catalan triangles.
\begin{defi}\label{defPC}
For $d\geq 2$, The triangle of numbers $\{c^{(d)}_{k,n}\}$, with $k\in\bZ$ and $n\in\bZ_{\geq0}$ is defined by the initial condition
\[c^{(d)}_{-1,0}=-1\,,\ \ \ \ c_{1,0}^{(d)}=1\,,\ \ \ \ \text{and}\ \ \ \ c_{i,0}^{(d)}=0\,,\ \forall i\neq-1,1\ ,\]
and the Pascal recurrence of order $d$:
\begin{equation}\label{higherPascalrule2}c_{k,n+1}^{(d)}=c_{k-(d-1),n}^{(d)}+c_{k-(d-3),n}^{(d)}+\dots+c_{k+d-3,n}^{(d)}+c_{k+d-1,n}^{(d)}\ .\end{equation}
We call the resulting triangle the \emph{Catalan triangle of order $d$} and denote it $\cPC_d$.
\end{defi}
Note that this is the same recurrence as for the Pascal triangles of the preceding section. Here is the triangle for $d=2$ where the small numbers on top shows the labelling of the columns:
\begin{equation}\label{PC2b}
\cPC_2=\ \  \begin{array}{ccccccccccccccccccc}
&\ldots&\scriptstyle{(-2)} & \scriptstyle{(-1)} & \scriptstyle{(0)}& \scriptstyle{(1)} & \scriptstyle{(2)} &  \ldots\\[0.6em]
&&& -1\  && 1 \\
&& -1\  && 0 && 1 \\
& -1 && -1\  && 1 && 1\\[-0.4em]
\reflectbox{$\ddots$} && -2\  && 0 && 2 && 1\\[-0.4em]
& \vdots && -2\  && 2 && 3 && 1 \\[-0.4em]
&&\vdots&& 0 && 5 && 4 && 1\\[-0.4em]
&&&\vdots&& 5 && 9 && 5 && 1\\
&&&& 0 && 14 && 14 && 6 && 1\\
&&&&& 14 && 28 && 20 && 7 && 1\\
&&&& 0 && 42 && \ldots && &&  && \ddots\\
&&&&& 42 && && \ldots &&  && &&
\end{array}
\end{equation}
Then the number $c^{(2)}_{k,n}$ is the value of the entry in line $n$ and column $k$ of this triangle. Note that for line $n$ with $n$ even, it is sufficient to display the columns with odd $k$, while for line $n$ with $n$ odd, it is sufficient to display the even columns. The (positive half of the) triangle is called Catalan triangle, see A008313, A008315, A009766, A053121 in \cite{OEIS}, and the numbers appearing in the first column are the famous Catalan numbers, A000108 in \cite{OEIS}. We abuse terminology and keep the name \emph{Catalan triangle} (of order 2) for the whole triangle in (\ref{PC2b}).

For $d=3$, here is what we call the Catalan triangle of order 3:
\begin{equation}\label{PC3b}
\cPC_3=\quad \begin{array}{ccccccccccccccccc}
&\ldots&\scriptstyle{(-3)} & \scriptstyle{(-1)} & \scriptstyle{(1)} & \scriptstyle{(3)} &  \ldots\\[0.6em]
&&& -1 & 1 \\
& & -1 & 0 & 0 & 1 \\
& -1 & -1 & -1 &1 & 1 & 1\\[-0.4em]
\reflectbox{$\ddots$}&&&-1 & 1 & 3 & 2 & 1\\[-0.4em]
&&\vdots&-3 & 3 & 6 & 6 & 3 & 1\\[-0.4em]
&\vdots&&-6 & 6 & 15 & 15 & 10 & 4 & 1\\
&&&-15& 15 & 36 & 40 & 29 & 15 & 5 & 1\\[-0.4em]
&&&-36& 36 & 91 & &\ldots &&&&\ddots\\
&&&-91& 91 & &\ldots 
\end{array}
\end{equation}
We show again the labels of the columns and we note that for $d$ odd, it is enough to show only the odd columns. The (positive half) of the triangle (\ref{PC3}) is A089942 in OEIS and the first positive column are the Riordan numbers, see A005043 \cite{OEIS}.

For $d=4$, what we call the Catalan triangle of order 4 looks like
\begin{equation*}\label{PC4}
\cPC_4=\ \begin{array}{ccccccccccccccccccccc}
&&&&\ldots&\scriptstyle{(-2)} & \scriptstyle{(-1)} & \scriptstyle{(0)}& \scriptstyle{(1)} & \scriptstyle{(2)} &  \ldots\\[0.6em]
&&&&&&-1&& 1 \\
&&&-1&&0 && 0 && 0 && 1\\
-1&&-1&&-1&&-1 && 1 && 1&& 1&& 1 \\[-0.4em]
&\vdots& &-4&& -2 &&0 && 2 && 4 && 3&&2&&1\\[-0.4em]
&&&&\vdots&& -4 && 4 && 9 && 11 && 10 && 6 && 3 && 1\\[-0.4em]
&&&\vdots& & -20\!\! && 0 && 20 && 34 && \ldots & && && &&\\
&&&&&& -34\!\! && 34 && \ldots
\end{array}
\end{equation*}
This triangle is not found in OEIS. Only the first column of positive coefficients appears in A264607 \cite{OEIS}. This is the same for higher $d$.

\paragraph{Equivalent definition and a formula.} In the following proposition we give the simple algebraic interpretation of the positive integers appearing in these triangles, along with an explicit formula in terms of the generalised binomial coefficients of the Pascal triangle $\cP_d$. The first item below can be seen as an equivalent definition of the positive integers $c_{k,n}^{(d)}$, $k>0$.
\begin{prop}\label{propPC}
Recall that $[k]_q=\frac{q^k-q^{-k}}{q-q^{-1}}=q^{1-k}+q^{3-k}+\dots+q^{k-3}+q^{k-1}$.
\begin{enumerate}\item We have
\begin{equation}\label{expansionPC}
[d]_q^n=\sum_{k>0}c_{k,n}^{(d)}[k]_q\ . 
\end{equation}
\item We have the formula in terms of generalised binomial coefficients:
\begin{equation}\label{formulaPC}
c_{k,n}^{(d)}=b_{k-1,n}^{(d)}-b_{k+1,n}^{(d)}\ . 
\end{equation}
\end{enumerate}
\end{prop}
\begin{proof}
It is immediate from their definition that the integers $c_{k,n}^{(d)}$ appear in the following expansion:
\[(q-q^{-1})[d]_q^n=\sum_{k\in\bZ}c_{k,n}^{(d)}q^k\ .\]
Indeed the initial condition corresponds to $q-q^{-1}$, while the recurrence corresponds to the multiplication by $[d]_q=q^{1-d}+q^{3-d}+\dots+q^{d-3}+q^{d-1}$. The polynomial $(q-q^{-1})[d]_q^n$ is antisymmetric through the exchange of $q$ and $q^{-1}$ so we have in fact
\[(q-q^{-1})[d]_q^n=\sum_{k>0}c_{k,n}^{(d)}(q^k-q^{-k})\ .\]
Dividing both sides by $q-q^{-1}$ yields the expansion formula in (\ref{expansionPC}).

Formula (\ref{formulaPC}) is almost immediate when looking at the triangle. Indeed there are two contributions to the value of $c_{k,n}^{(d)}$, both of them coming from a Pascal rule of order $d$ applied to one of the initial vertices. The first one comes with a $+$ sign from the initial vertex $+1$ and is equal to $b_{k-1,n}^{(d)}$ (since the initial vertex $+1$ is in column $1$ while it is in column $0$ in the Pascal triangle). Similarly, the second contribution comes with a $-$ sign from the initial vertex $-1$ and is equal to $-b_{k+1,n}^{(d)}$. 
\end{proof}

\paragraph{Example.} For $d=2,3$, the integers $c_{k,n}^{(2)}$ and $c_{k,n}^{(3)}$ appearing in the Catalan triangles $\cPC_2$ and $\cPC_3$ above satisfy
\[(q+q^{-1})^n=\sum_{k>0}c_{k,n}^{(2)}[k]_q\ \ \ \ \text{and}\ \ \ \ (q^2+1+q^{-1})^n=\sum_{k>0}c_{k,n}^{(3)}[k]_q\ .\]
In particular, the Catalan numbers, which are the coefficients $c_{1,2n}^{(2)}$ appear as the coefficients in front of $[1]_q=1$ when $(q+q^{-1})^{2n}$ is expanded in terms of the $q$-numbers $[k]_q$. The Riordan numbers, which are the coefficients $c_{1,n}^{(3)}$, appear as the coefficients in front of $[1]_q=1$ in the $q$-numbers expansion of $(q^2+1+q^{-1})^n$.

The entries in the triangle $\cPC_2$ are all given in terms of binomial coefficients as:
\[c_{k,n}^{(2)}=\binom{n}{\frac{k-1+n}{2}}-\binom{n}{\frac{k+1+n}{2}}\]
and in particular this recovers the well-known formula for the Catalan numbers:  $c_{1,2n}^{(2)}=\binom{2n}{n}-\binom{2n}{n+1}$.

\subsection{Multiplicities in tensor products}

We consider the Lie algebra $sl_2$ and we denote by $\rep{d}$ its unique $d$-dimensional irreducible representation. For example, the representation $\rep{2}$ is the defining (or vector) representation where a matrix in $sl_2$ is simply sent to itself. One possible construction of $\rep{d+1}$ is as the symmetrised $d$-th power of the defining representation $\rep{2}$ (see for example \cite{FH91} for the well-known facts about $sl_2$ that we are using here and below).
\begin{rema}
The representation $\rep{d}$ can also be seen as a representation of the Lie group $SU(2)$ and in this context it is often called the representation of spin $\frac{d-1}{2}$ so that for example the $2$-dimensional representation $\rep{2}$ is the one of spin $\frac{1}{2}$.
\end{rema}
The Cartan element of $sl_2$ is the diagonal matrix $H=\left(\begin{array}{cc} 1 & 0\\ 0 & -1\end{array}\right)$. From $H$, we can define the character of any finite-dimensional representation of $sl_2$. If $V$ is such a representation of dimension $n$ and $\rho_V(H)$ is the matrix representing $H$, we define:
\[ch_V=\sum_{i=1}^n q^{\lambda_i}\ \ \ \ \ \text{where}\ \{\lambda_1,\dots,\lambda_n\}=Sp(\rho_V(H))\,.\]
Thus the character collects the eigenvalues (counted with multiplicities) of the action of $H$. In particular, we have
\[ch_{\rep{2}}=q+q^{-1}=[2]_q\,,\]
and more generally, for example from the description of $\rep{d}$ as a symmetrised power, it is easy to see that
\[ch_{\rep{d}}=q^{d-1}+q^{d-3}+\dots+q^{3-d}+q^{1-d}=[d]_q\ .\]
Now let us say we want to find the decomposition of a tensor power of the representation $\rep{d}$ as a direct sum of irreducible representations:
\[\rep{d}^{\otimes n}\cong\bigoplus_{k} \alpha^{(d)}_{k,n}\rep{k}\ .\]
It turns out that a finite-dimensional representation of $sl_2$ is uniquely determined by its character and therefore it is enough to work out the expansion of the character of $\rep{d}^{\otimes n}$ in terms of the characters of the representations $\rep{k}$, namely we have that the multiplicities $\alpha^{(d)}_{k,n}$ are determined by the identity
\[ch_{\rep{d}^{\otimes n}}=\sum_{k>0}\alpha_{k,n}^{(d)}ch_\rep{k}\ .\]
Of course, the left hand side is $(ch_{\rep{d}})^n=([d]_q)^n$ while the right hand side is its expansion in terms of the $q$-numbers $[k]_q$. Therefore, the first item of Proposition \ref{propPC} translates into the following representation-theoretic interpretation of the Catalan triangles.
\begin{coro}
We have
\[\rep{d}^{\otimes n}\cong\bigoplus_{k} c^{(d)}_{k,n}\rep{k}\ ,\]
where the numbers $c^{(d)}_{k,n}$ are defined in Definition \ref{defPC} and calculated in (\ref{formulaPC}).
\end{coro}

\begin{rema}
The integers $c_{k,n}^{(d)}$ being the multiplicities appearing in the tensor product $\rep{d}^{\otimes n}$, they are also the dimensions of the irreducible representations of the centraliser of $\rep{d}^{\otimes n}$. For $d=2$, this centraliser is the Temperley--Lieb algebra. For arbitrary $d$, this centraliser is a quotient of the algebra called in \cite{CP23} the algebra of fused permutations (or the fused Hecke algebra in the quantum group setting). Its representations were described in terms of Young diagrams and semistandard tableaux in \cite{CP23}, and here is the beginning of the Bratelli diagram for $d=3$:
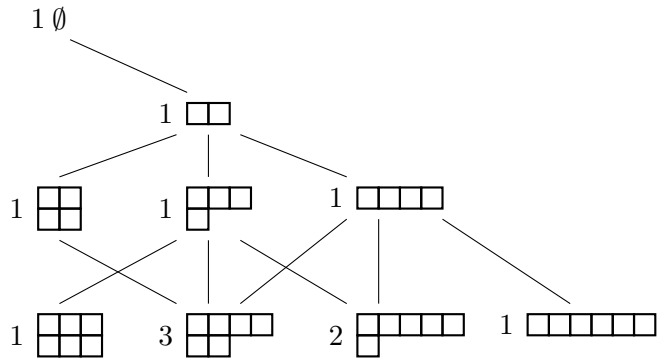
\begin{center}
 \begin{tikzpicture}[scale=0.28]

\node at (-1,0) {$1$};
\node at (0,0) {$\emptyset$};

\draw (0.5,-1) -- (6,-3.5);
\diag{6}{-4}{2};\node at (5,-4.5) {$1$};

\draw (5.5,-5.5) -- (0,-7.5);
\draw (7,-5.5) -- (7,-7.5);
\draw (8.5,-5.5) -- (13.5,-7.5);
\diagg{-1}{-8}{2}{2};\node at (-2,-9) {$1$};
\diagg{6}{-8}{3}{1};\node at (5,-9) {$1$};
\diag{14}{-8}{4};\node at (13,-8.5) {$1$};

\draw (0,-10.5) -- (5.5,-13.5);
\draw (7,-10.5) -- (7,-13.5);
\draw (13.5,-9.5) -- (8.5,-13.5);
\draw (5.5,-10.5) -- (0,-13.5);
\draw (8.5,-10.5) -- (13.5,-13.5);
\draw (15,-9.5) -- (15,-13.5);
\draw (18,-9.5) -- (24,-13.5);
\diagg{-1}{-14}{3}{3};\node at (-2,-15) {$1$};
\diagg{6}{-14}{4}{2};\node at (5,-15) {$3$};
\diagg{14}{-14}{5}{1};\node at (13,-15) {$2$};
\diag{22}{-14}{6};\node at (21,-14.5) {$1$};
\end{tikzpicture}
\end{center}
The rule (called the Pieri rule) is that we consider only partitions with at most two rows (this is for $sl_2$) and we add $d-1$ boxes at each step with no two boxes in the same column. For example, \begin{tikzpicture}[scale=0.28]\diagg{0}{0}{2}{2};\end{tikzpicture} is not connected to \begin{tikzpicture}[scale=0.28]\diagg{0}{0}{3}{3};\end{tikzpicture}. The integers appearing at each vertex reproduce the Catalan triangle $\cPC_3$. They count the number of paths, or equivalently a certain number of semistandard Young tableaux. For example, the number $3$ in the example corresponds to the following three semistandard tableaux:
\[\begin{array}{cccc}
\fbox{\scriptsize{$1$}} & \hspace{-0.35cm}\fbox{\scriptsize{$1$}}& \hspace{-0.35cm}\fbox{\scriptsize{$3$}}& \hspace{-0.35cm}\fbox{\scriptsize{$3$}} \\[-0.2em]
\fbox{\scriptsize{$2$}} & \hspace{-0.35cm}\fbox{\scriptsize{$2$}}
\end{array}\ ,\qquad\quad \begin{array}{cccc}
\fbox{\scriptsize{$1$}} & \hspace{-0.35cm}\fbox{\scriptsize{$1$}}& \hspace{-0.35cm}\fbox{\scriptsize{$2$}}& \hspace{-0.35cm}\fbox{\scriptsize{$3$}} \\[-0.2em]
\fbox{\scriptsize{$2$}} & \hspace{-0.35cm}\fbox{\scriptsize{$3$}}
\end{array}\ ,\qquad\quad \begin{array}{cccc}
\fbox{\scriptsize{$1$}} & \hspace{-0.35cm}\fbox{\scriptsize{$1$}}& \hspace{-0.35cm}\fbox{\scriptsize{$2$}}& \hspace{-0.35cm}\fbox{\scriptsize{$2$}} \\[-0.2em]
\fbox{\scriptsize{$3$}} & \hspace{-0.35cm}\fbox{\scriptsize{$3$}}
\end{array}\ .\]
\end{rema} 
\subsection{Sums of squares}

In the triangles $\cPC_d$, for every $d\geq 2$, similarly as for the Pascal triangles $\cP_d$, we have the following remarkable equality.
\begin{prop}\label{propsquaresPC}
We have:
\begin{equation}\label{sumofsquaresPC}
c_{1,2n}^{(d)}=\sum_{k>0}(c_{k,n}^{(d)})^2\ .
\end{equation}
In words, the first positive coefficient of line $2n$ is equal to the sum of the squares of the positive coefficients of line $n$.
\end{prop}
As can be checked in the examples above, in the Catalan triangle $\cPC_2$, the sum of the squares of the positive half of line $n$ gives the sequence $1,1,2,5,14,42,132,...$ of Catalan numbers (A000108 in \cite{OEIS}). 

For $d=3$, this sum of squares gives the sequence $1,1,3,15,91,603,...$ (A099251 in \cite{OEIS}). Note also that the first positive coefficients of all lines (not only the even ones) gives the sequence of Riordan numbers $1,0,1,1,3,6,15,36,91...$ (A005043 in \cite{OEIS}).

For $d=4$, this sum of squares gives the sequence $1,1,4,34,364...$ which is A264607 in \cite{OEIS}.

\paragraph{A representation-theoretic proof.} The representation-theoretic interpretation of the ``sum of squares'' formula in Proposition \ref{propsquaresPC} is very similar to the one given in the preceding section for the Pascal triangles. Namely it relies on the same formal generalities:
\begin{equation}\label{soqrepproof2}\begin{array}{rcl}
\textrm{End}_{sl_2}(\rep{d}^{\otimes n}) & = & \textrm{Hom}_{sl_2}(\rep{d}^{\otimes n},\rep{d}^{\otimes n})\\[0.4em]
 & \cong & \textrm{Hom}_{sl_2}(\rep{d}^{\otimes n}\otimes (\rep{d}^{\star})^{\otimes n},\rep{1})\\[0.4em]
 & \cong & \textrm{Hom}_{sl_2}(\rep{d}^{\otimes 2n},\rep{1})\ .
\end{array}\end{equation}
Note here that $\rep{1}$ is indeed the trivial representation of $sl_2$. The only key fact (which is typical of $sl_2$) is that the irreducible representations $\rep{k}$ are self-contragredient. This can be seen easily since in the contragredient representation $\rep{k}^{\star}$, the action of $H$ is replaced by the action of $-H$. Therefore, the character of $\rep{k}^{\star}$ is the character of $\rep{k}$ with $q$ replaced by $q^{-1}$, but the character of $\rep{k}$ is symmetric in $q$ and $q^{-1}$, hence we have $ch_{\rep{k}^{\star}}=ch_{\rep{k}}$ and therefore $\rep{k}^{\star}\cong \rep{k}$.

As in the preceding section, the dimension of $\textrm{End}_{sl_2}(\rep{d}^{\otimes n})$ is the sum of the squares of the multiplicities, that is, of the positive entries in line $n$, while the dimension of $\textrm{Hom}_{sl_2}(\rep{d}^{\otimes 2n},\rep{1})$ is exactly the first entry in line $2n$. Thus taking the dimensions in the isomorphism (\ref{soqrepproof2}) gives the equality (\ref{sumofsquaresPC}).

\paragraph{A combinatorial proof.} Here we give a combinatorial proof of Formula (\ref{sumofsquaresPC}) independent of representation theory and in the same spirit as the proof for the Pascal triangle.

As before, we will see our triangles as oriented graphs, by connecting an entry to the $d$ entries below it to which it contributes (the edges are always oriented downward). Let us denote by $|p_{x\to y}|$ the number of paths following these oriented edges from a vertex $x$ to a vertex $y$. We will denote $\mathbf{1^+}$ and $\mathbf{1^-}$ the two initial vertices of the graph. Then note that, by the very definition of the triangles, the value of an entry at a certain vertex $v$ is equal to the number of paths from $\mathbf{1^+}$ to $v$ minus the number of paths from $\mathbf{1^-}$ to $v$.

From now on, we will denote $X$ the vertex in position $(1,2n)$ and $x_k$ the vertices $(k,n)$ of line $n$, where $k\in\mathbb{Z}$. We make the following calculation
\[\begin{array}{rcl}
c^{(d)}_{1,2n} & = & |p_{\mathbf{1^+}\to X}|-|p_{\mathbf{1^-}\to X}|\\
 & = & \displaystyle \sum_{k\in\bZ} \Bigl(|p_{\mathbf{1^+}\to x_k}||p_{x_k\to X}|-|p_{\mathbf{1^-}\to x_k}||p_{x_k\to X}| \Bigr)\\
 & = & \displaystyle \sum_{k\in\bZ} \Bigl(|p_{\mathbf{1^+}\to x_k}|^2-|p_{\mathbf{1^-}\to x_k}||p_{\mathbf{1^+}\to x_k}| \Bigr)\ .
\end{array}\]
In the second line we have been using that a path from $\mathbf{1^+}$ to $X$ is the concatenation of a path from $\mathbf{1^+}$ to $x_k$ (for some $k$) and a path from $x_k$ to $X$; and similarly when we start from $\mathbf{1^-}$. In the third line, we used that by symmetry we have $|p_{x_k\to X}|=|p_{\mathbf{1^+}\to x_k}|$ for any vertex $x_k$ on line $n$.

Now we use the symmetry $|p_{\mathbf{1^+}\to x_k}|=|p_{\mathbf{1^-}\to x_{-k}}|$ valid for any $k\in\bZ$, to rewrite everything in terms of $x_k$ with positive $k$. We get:
\[c^{(d)}_{1,2n}=\sum_{k>0} \Bigl(|p_{\mathbf{1^+}\to x_k}|^2-|p_{\mathbf{1^-}\to x_k}||p_{\mathbf{1^+}\to x_k}| \Bigr)+\sum_{k>0} \Bigl(|p_{\mathbf{1^-}\to x_k}|^2-|p_{\mathbf{1^+}\to x_k}||p_{\mathbf{1^-}\to x_k}| \Bigr)\ .\]
We see that this is the same as:
\[\sum_{k>0}(c_{k,n}^{(d)})^2=\sum_{k>0}\bigl(|p_{\mathbf{1^+}\to x_k}|-|p_{\mathbf{1^-}\to x_k}|\bigr)^2\ ,\]
which concludes the combinatorial verification of (\ref{sumofsquaresPC}).

\section{Motzkin triangles}\label{sec_Mot}

\subsection{Definition and examples}\label{defPM}

We start with the formal definition of the Motzkin triangle and its higher analogues.
\begin{defi}
For $d\geq 2$, the triangle of numbers $\{m^{(d)}_{k,n}\}$, with $k\in\bZ$ and $n\in\bZ_{\geq0}$ is defined by the initial condition
\[m^{(d)}_{-1,0}=-1\,,\ \ \ \ m_{1,0}^{(d)}=1\,,\ \ \ \ \text{and}\ \ \ \ m_{i,0}^{(d)}=0\,,\ \forall i\neq-1,1\ ,\]
and the following recurrence:
\begin{equation}\label{higherPascalrule3}m_{k,n+1}^{(d)}=m_{k-(d-1),n}^{(d)}+m_{k-(d-2),n}^{(d)}+\dots+m_{k+d-2,n}^{(d)}+m_{k+d-1,n}^{(d)}\ .\end{equation}
We call the resulting triangle the \emph{Motzkin triangle of order $d$} and denote it $\cPM_d$.
\end{defi}
We emphasise that the recurrence is different from before: now we go by steps of $1$ in the $k$-indices instead of steps of $2$. It corresponds to multiplication by the polynomial $q^{-(d-1)}+q^{-(d-2)}+\dots+q^{d-2}+q^{d-1}$ (see Proposition \ref{propPM} below).

Here we show the simplest example, for $d=2$, with the labels of the columns indicated:
\begin{equation}\label{PM2b}
\cPM_2=\quad\begin{array}{ccccccccccccccccc}
&\ldots&\scriptstyle{(-2)} & \scriptstyle{(-1)} & \scriptstyle{(0)}& \scriptstyle{(1)} & \scriptstyle{(2)} &  \ldots\\[0.6em]
&&& -1 &0& 1 &&&&&&&& \\
&& -1 & -1 & 0 & 1 & 1 &&&&&&& \\
&-1&-2&-2& 0 & 2 & 2 & 1 &&&&&&  \\[-0.5em]
\reflectbox{$\ddots$}&&\vdots& & 0 & 4 & 5 & 3 & 1&&&&& \\[-0.5em]
&&&&\vdots& 9 & 12 & 9 & 4 & 1&&&& & \\[-0.5em]
&&&&\vdots& 21 & 30 & 25 & 14 & 5 & 1\\[-0.5em]
&&&&& 51 & 76 & \ldots &&&& \ddots
\end{array}
\end{equation}
For $d=2$ the recurrence contains three terms: $m_{k,n+1}^{(2)}=m_{k-1,n}^{(2)}+m_{k,n}^{(2)}+m_{k+1,n}^{(2)}$, so the triangle is given by a Pascal rule of order $3$. The first positive column consists of the Motzkin numbers $m_{1,n}^{(2)}=1,1,2,4,9,21,51,...$, see A001006 in \cite{OEIS}, the positive half being referred to as the Motzkin triangle, see A026300 or A064189 in OEIS \cite{OEIS}.
\begin{rema}
For arbitrary $d$, the recurrence contains $2d-1$ terms, so the triangle $\cPM_d$ is given by a Pascal rule of order $2d-1$ (always an odd number) from the inital condition $-1,0,1$.
\end{rema}
The next example is for $d=3$:
\begin{equation}\label{PM3}
\cPM_3=\quad\begin{array}{ccccccccccccccccc}
&\ldots&\scriptstyle{(-2)} & \scriptstyle{(-1)} & \scriptstyle{(0)}& \scriptstyle{(1)} & \scriptstyle{(2)} &  \ldots\\[0.6em]
&&& -1 &0& 1 &&&&& \\
& -1 & -1 &0& 0 &0& 1 & 1 &&& \\[-0.5em]
\reflectbox{$\ddots$}&&-2& -2 & 0 & 2 & 2 & 2 & 2& 1\\[-0.5em]
&&\vdots& -4 & 0 & 4 & 8 & 9 & 7&5 & 3 & 1  \\[-0.5em]
&&&\vdots&0& 17 & 28 & 33 & 32 & 25 & 16 & 9 & 4 & 1\\[-0.5em]
&&&&\vdots& 61 & 110 & 135 & \ldots &&&&&& \ddots\\
&&&&& 245
\end{array}
\end{equation}
Here, each entry is the sum of the 5 entries above it, for example $17=-4+0+4+8+9$. The sequence in the first column $m_{1,n}^{(3)}=1,0,2,4,17,61,245,...$ is an analogue of the Motzkin sequence and does not seem to be in OEIS (and neither is the full triangle).

\paragraph{Equivalent definition and formulas.} We give the simple algebraic interpretation of the positive integers appearing in the triangles $\cPM_d$, along with an explicit formula in terms of the generalised binomial coefficients and in terms of the coefficients of the Catalan triangles. The first item below can be seen as an equivalent definition of the positive integers $m_{k,n}^{(d)}$, $k>0$.
\begin{prop}\label{propPM}
\begin{enumerate}\item We have
\begin{equation}\label{expansionPM}
([d-1]_q+[d]_q)^n=\sum_{k>0}m_{k,n}^{(d)}[k]_q\ . 
\end{equation}
\item We have the formula in terms of generalised binomial coefficients:
\begin{equation}\label{formulaPM1}
m_{k,n}^{(d)}=b_{2k-2,n}^{(2d-1)}-b_{2k+2,n}^{(2d-1)}\ . 
\end{equation}
\item We have the formula in terms of coefficients of Catalan triangles:
\begin{equation}\label{formulaPM2}
m_{k,n}^{(d)}=c_{2k-1,n}^{(2d-1)}+c_{2k+1,n}^{(2d-1)}\ . 
\end{equation}
\end{enumerate}
\end{prop}
\begin{proof}
The first two items are proved in the same way as for the Catalan triangles in Proposition \ref{propPC}. One only has to notice for item 1 that:
\[q^{-(d-1)}+q^{-(d-2)}+\dots+q^{d-2}+q^{d-1}=[d-1]_q+[d]_q\ .\]
Note also that the factor $2$ in the $k$-indices in Formula (\ref{formulaPM1}) comes from the fact that the columns of the Pascal triangle $\cP_{2d-1}$ are indexed by even integers while the columns of $\cPM_d$ are indexed by all integers.

Finally, item 3 follows immediately from item 2 and Formula (\ref{formulaPC}) in Proposition \ref{propPC}, which we recall is saying that $c_{k,n}^{(d)}=b_{k-1,n}^{(d)}-b_{k+1,n}^{(d)}$.
\end{proof}

\paragraph{Example.} For $d=2$, the integers $m_{k,n}^{(2)}$ appearing in the Motzkin triangle $\cPM_2$ satisfy
\[(q+1+q^{-1})^n=\sum_{k>0}m_{k,n}^{(2)}[k]_q\ .\]
In particular, the Motzkin numbers, which are the coefficients $m_{1,2n}^{(2)}$, appear as the coefficients in front of $[1]_q=1$ when $(q+1+q^{-1})^{n}$ is expanded in terms of the $q$-numbers $[k]_q$.

The entries in the triangle $\cPM_2$ are all given in terms of trinomial coefficients $b_{k,n}^{(3)}$ and in terms of numbers from the Catalan triangle of order $3$. In particular, we have:
\[m_{1,n}^{(2)}=b_{0,n}^{(3)}-b_{4,n}^{(3)}=c_{1,n}^{(3)}+c_{3,n}^{(3)}\,.\]
In words, entries of the Motzkin triangles are differences of next to consecutive coefficients of the Pascal triangle $\cP_3$ of order $3$, and are also sums of consecutive entries in the Catalan triangle $\cPC_3$ of order $3$.

For $d=3$, the formula relating the entries of $\cPM_3$ with the entries of $\cPC_5$ can be checked looking at (\ref{PM3}) above together with
\begin{equation}\label{PC5}
\cPC_5=\quad\begin{array}{ccccccccccccccccc}
&\ldots&\scriptstyle{(-3)}\!\! & \scriptstyle{(-1)} & \scriptstyle{(1)} & \scriptstyle{(3)} &  \ldots\\[0.6em]
&&& -1 & 1 &&&&& \\
& -1 &0 & 0 &0& 0 & 1 &&& \\[-0.5em]
\reflectbox{$\ddots$}&&\ldots\!\!& -1 & 1 & 1 & 1 & 1 & 1\\[-0.5em]
&&&\vdots & 1 & 3 & 5 & 4 & 3&2 & 1 &   \\[-0.5em]
&&&\vdots&  5 & 12 & 16 & 17 & 15 & 10 & 6 & 3 & 1\\[-0.5em]
&&&&16 & 45 && \ldots& & & & &&\ddots
\end{array}
\end{equation}

\subsection{Multiplicities in tensor products}
As for the Catalan triangles, the positive entries of the Motzkin triangles are multiplicities appearing in tensor products of $sl_2$-representations. Indeed we have:
\[ch_{\rep{d-1}}+ch_{\rep{d}}=[d-1]_q+[d]_q=q^{d-1}+q^{d-2}+\dots+q^{2-d}+q^{1-d}\ .\]
Therefore, the first item of Proposition \ref{propPM} translates immediately into the following representation-theoretic interpretation of the Motzkin triangles.
\begin{coro}
We have
\[(\rep{d-1}\oplus\rep{d})^{\otimes n}\cong\bigoplus_{k} m^{(d)}_{k,n}\rep{k}\ .\]
\end{coro}

\begin{rema}
For $d=2$, the positive entries in the line $n$ of the Motzkin triangle $\cPM_2$ are the multiplicities appearing in $(\rep{1}\oplus\rep{2})^{\otimes n}$. These multiplicities are the dimensions of the irreducible representations of the centraliser of this tensor product. Such a centraliser was studied in \cite{BH14} under the name of \emph{Motzkin algebra}. Its dimension is given by the ``sum of squares'' formula, see below, and is equal to the first positive coefficient $m_{1,2n}^{(2)}$ of line $2n$, namely the $2n$-th Motzkin number. For arbitrary $d$, the entries of the Motzkin triangle $\cPM_d$ can be seen as the dimension of the irreducible representations of centralisers of $(\rep{d-1}\oplus\rep{d})^{\otimes n}$, which can be seen as algebras generalising the Motzkin algebras of \cite{BH14}.
\end{rema} 

\subsection{Sums of squares}

In the triangles $\cPM_d$, for every $d\geq 2$, similarly as for the Pascal and the Catalan triangles, we have the following remarkable equality.
\begin{prop}\label{propsquaresPM}
We have:
\begin{equation}\label{sumofsquaresPM}
m_{1,2n}^{(d)}=\sum_{k>0}(m_{k,n}^{(d)})^2\ .
\end{equation}
In words, the first positive coefficient of line $2n$ is equal to the sum of the squares of the positive coefficients of line $n$.
\end{prop}
\begin{proof}
Both the combinatorial and representation-theoretic proofs used for the Catalan triangles applies in exactly the same way here. 
\end{proof}
For example, taking the fourth line ($n=3$) of the Motzkin triangle $\cPM_2$ in (\ref{PM2b}), we have $4^2+8^2+9^2+7^2+5^2+3^2+1^2=245$, which is indeed the first positive coefficient of line seven ($n=6$).

\paragraph{A general ``sum of squares'' phenomenon.} The ``sum of squares'' phenomenon that we have seen appearing several times now  can be formulated at the following general level.
\begin{prop}
For any Laurent polynomial $P\in\bC[q,q^{-1}]$ satisfying $P(q)=P(q^{-1})$, define coefficients $a_{k,n}$ by
\[P^n=\sum_{k>0}a_{k,n}[k]_q\,.\]
Then we have for any $n\geq 0$,
\[\sum_{k>0}(a_{k,n})^2=a_{1,2n}\ .\]
\end{prop}
\begin{proof}
First note that any Laurent polynomial $P\in\bC[q,q^{-1}]$ satisfying $P(q)=P(q^{-1})$ can indeed be expanded uniquely in terms of the $q$-numbers $[k]_q$: take the coefficient $a$ of the leading term $q^N$ of $P$, write $P=a[N]_q+P'$ and reiterate for $P'$ and so on.

The representation-theoretic proof that we used in the preceding sections is perfectly available here and as simple as before. Indeed, one can see the polynomial $P$ as an $sl_2$-character:
\[P=ch_V\,,\ \ \ \ \ \ \text{where $V=\bigoplus_{k>0} a_{k,1}\rep{k}$.}\]
Then the integers $a_{k,n}$ are the multiplicities appearing in the tensor product $V^{\otimes n}$, while $a_{1,2n}$ is the multiplicity of the trivial representation in $V^{\otimes 2n}$. The reasoning of Section \ref{sec_Cat} can therefore be repeated without change.
\end{proof}
As a concluding example, take
\[P=q^2+q+3+q^{-1}+q^{-2}=2[1]_q+[2]_q+[3]_q\ .\]
We will calculate $P^n$ from a triangle of numbers using the same trick as before, namely, by writing:
\[(q-q^{-1})P^n=\sum_{k>0}a_{k,n}(q^k-q^{-k})\ .\]
Therefore, the coefficients $a_{k,n}$ are found in the following triangle
\begin{equation}\label{finalexample}
\begin{array}{ccccccccccccccccc}
&\ldots&\scriptstyle{(-2)} & \scriptstyle{(-1)} & \scriptstyle{(0)}& \scriptstyle{(1)} & \scriptstyle{(2)} &  \ldots\\[0.6em]
&&& -1 &0& 1 &&&&& \\
& -1 & -1 &-2& 0 &2& 1 & 1 &&& \\[-0.5em]
\reflectbox{$\ddots$}&&& -6 & 0 & 6 & 6 & 6 & 2& 1\\[-0.5em]
&&&\vdots & 0 & 24 & 32 & 33 & 19&11 & 3 & 1  \\[-0.5em]
&&&&\vdots & 113 & & \ldots &&&&& \ddots
\end{array}
\end{equation}
The recursive rule for going from one line to the next corresponds to $P$ and is of the form \begin{tikzpicture}[scale=0.25]
\draw (0,0) -- (0,2);
\node at (0,1){$\scriptstyle{3}$};
\draw (-0.5,0) -- (-2,2);
\node at (-1.25,1){$\scriptstyle{1}$};
\draw (-1,0) -- (-4,2);
\node at (-2.5,1){$\scriptstyle{1}$};
\draw (0.5,0) -- (2,2);
\node at (1.25,1){$\scriptstyle{1}$};
\draw (1,0) -- (4,2);
\node at (2.5,1){$\scriptstyle{1}$};
\end{tikzpicture}
(which is not a rule of a Pascal triangle). For example, $24=-6+0+3\times6+6+6$. The triangle gives the multiplicities in $V^{\otimes n}$ with $V=\rep{1}\oplus\rep{1}\oplus\rep{2}\oplus\rep{3}$. One can check the sum of squares with $2^2+1^2+1^2=6$ and $6^2+6^2+6^2+2^2+1^2=113$.

\end{document}